\documentclass[12pt]{amsart}
\usepackage{amssymb,hyperref}
\usepackage{latexsym}
\usepackage{amsmath}

\theoremstyle{plain}
\everymath{\displaystyle}
\newtheorem{theorem}{Theorem}[section]

\newtheorem{lemma}[theorem]{Lemma}

\theoremstyle{definition}
\newtheorem{question}[theorem]{Question}
\newtheorem{definition}[theorem]{Definition}
\newtheorem{remark}[theorem]{Remark}
\numberwithin{equation}{section}

\newcommand{\m}{\mathfrak m}

\begin{document}
\title{The Brian\c{c}on-Skoda theorem  and coefficient ideals for non $\m$-primary ideals} 
\author{Ian M. Aberbach and Aline Hosry}
\address{Department of Mathematics \\
 	University of Missouri\\
	Columbia, MO 65211, USA}
\email{aberbachi@missouri.edu}
\address{Department of Mathematics \\
 	University of Missouri\\
	Columbia, MO 65211, USA}
\email{aline.hosry@mizzou.edu}

\date{\today}
\maketitle
\begin{abstract}
We generalize a Brian\c{c}on-Skoda type theorem  first studied by Aberbach and Huneke. With some conditions on a regular local ring $(R,\m)$ containing a field, and an ideal $I$ of $R$ with analytic spread $\ell$ and a minimal reduction $J$, we prove that for all $w \geq -1$, $ \overline{I^{\ell+w}} \subseteq J^{w+1} \mathfrak{a} (I,J),$ where $\mathfrak{a}(I,J)$ is the coefficient ideal of $I$ relative to $J$, i.e. the largest ideal $\mathfrak{b}$ such that $I\mathfrak{b}=J\mathfrak{b}$.  Previously, this result was known only for $\m$-primary ideals.\\
\end{abstract}

\section{Introduction}

Throughout this paper all rings are assumed to be commutative, Noetherian and with identity. 

The classical Brian\c con-Skoda theorem, proved first by Brian\c con and Skoda in the complex analytic case \cite{BS}, and by Lipman and Sathaye for regular rings in general \cite{LS}, states that if $(R,\m)$ is a regular local ring, then given an ideal $I$ of analytic spread $\ell$, and a reduction $J$ of $I$,  we have $\overline{I^{\ell+w}} \subseteq J^{w+1}$ for $w\ge 0$.  Further refinements of this theorem have abounded.  One such refinement is  (see Section~\ref{sec:defs} for the definition of the coefficient ideal ${\mathfrak{a}} (I,J)$):

\begin{theorem} \label{coeffth1} \textup{(\cite{AH3}, Theorem 2.7)} Let $(R,\m)$ be a regular local ring of dimension $d$ containing a field and having an infinite residue field. Let $I$ be an $\m$-primary ideal and let $J$ be a minimal reduction of $I$. Then for all $w \geq -1$, 
$$ \overline{I^{d+w}} \subseteq J^{w+1} {\mathfrak{a}} (I,J).$$
\end{theorem}

Note that this theorem applies only to $\m$-primary ideals $I$.  The reason is that the proof relies on an iteration giving a descending sequence of ideals, all of which contain a fixed power of $I$.  Thus, in the $\m$-primary case, this descending sequence stabilizes, and the stable value is shown to be the desired value.  Therefore, the same proof will not work in the non-$\m$-primary case.  The main result of this paper (see Theorem~\ref{mainthm}) extends Theorem \ref{coeffth1} to regular rings where a certain quotient (depending on $I$) is complete---in particular, we show that the theorem is true for all ideals when $R$ itself is complete.


There have been a number of results of this type. Some of them are in \cite{AH1},\cite{AH2},\cite{AH3},\cite{AHT},\cite{HH},\cite{LT},\cite{Sw}. In particular, with the development of the theory of \textit{tight closure} by Hochster and Huneke, these authors proved a generalized Brian\c{c}on-Skoda theorem from which the original Brian\c{c}on-Skoda theorem could follow. We discuss this for rings containing a field in the next section, after the definition of tight closure.\\

\section{Integral closure, tight closure and theorems of Brian\c con-Skoda type}\label{sec:defs}

Recall that an element $x$ of $R$ is \textit{integral} over an ideal $I$ of $R$ if there exists a positive integer $k$ such that $x^k+a_1x^{k-1}+\dots+a_k=0$ where $a_i \in I^i$ for $1\leq i \leq k$. The set of all elements of $R$ that are integral over $I$ is an ideal of $R$ called the \textit{integral closure} of $I$. 

Another definition is the one of a \textit{reduction} of an ideal that was first introduced by Northcott and Rees \cite{NR}. An ideal $J \subseteq I$ is a \textit{reduction} of $I$ if there exists an integer $r$ such that $J I^r= I^{r+1}$. The least such integer is the \textit{reduction number} of $I$ with respect to $J$. A reduction $J$ of $I$ is called a \textit{minimal reduction} if $J$ is minimal with respect to inclusion among reductions. When the ring $(R,m)$ is local with infinite residue field, every minimal reduction $J$ of $I$ has the same number of minimal generators. This number is called the \textit{analytic spread} of $I$, denoted by $\ell(I)$, and we always have that ht($I$) $\leq \ell(I) \leq $ dim $R$. 
 If an ideal $J\subseteq I$ is a reduction, then $\overline{J}=\overline{I}$.

Let $R$ be a Noetherian ring of prime characteristic $p>0$ and let $q$ be a varying power of $p$. Let $R^o$ be the complement of the union of the minimal primes of $R$ and let $I$ be an ideal of $R$. Define $I^{[q]}=(i^q: i \in I)$, the ideal generated by the $q^{th}$ powers of all the elements of $I$. The \textit{tight closure} of $I$ is the ideal $I^*=\{x \in R;\, \mbox{for some}\; c \in R^o, cx^q \in I^{[q]},\; \mbox{for}\, q \gg 0\}$. We always have that $I \subseteq I^*\subseteq \overline{I}$. If $I^*=I$ then the ideal $I$ is said to be \textit{tightly closed}. A ring in which every ideal is tightly closed is called \textit{weakly F-regular}. We say that elements $x_1,\dots, x_n$ of $R$ are \textit{parameters} if the height of the ideal generated by them is at least $n$ (we allow this ideal to be the whole ring, in which case the height is said to be $\infty$). The ring $R$ is said to be \textit{F-rational} if the ideals generated by parameters are tightly closed.

The theory of tight closure gives another proof of the Brian\c{c}on-Skoda theorem in characteristic $p$.

\begin{theorem} \textup{(\cite{HH}, Theorem 5.4)} \label {HH} Let $R$ be a Noetherian ring of characteristic $p$, and let $I$ be an ideal of positive height generated by $n$ elements. Then for every $w \in \mathbb{N}$, $\overline{I^{n+w}} \subseteq (I^{w+1})^*$. In particular, $\overline{I^n} \subseteq I^*.$

If $R$ is weakly F-regular (in particular, if $R$ is regular), of characteristic $p$, then $\overline{I^{n+w}} \subseteq I^{w+1}$ and $\overline{I^n} \subseteq I.$
\end{theorem}

It should be noted that the characteristic zero case of the original Brian\c{c}on-Skoda theorem can be reduced to the characteristic $p$ case, but tight closure does not seem to offer such a generalization for rings of mixed characteristic.

Another theorem, established by Aberbach and Huneke, allows us to replace the assumption weakly F-regular by F-rational, in the second part of Theorem \ref{HH}. It states the following:

\begin{theorem} \textup{(\cite{AH1}, Theorem 3.6)} Let $(R,\m)$ be an F-rational local ring of characteristic $p$, and let $I \subseteq R$ be an ideal generated by $\ell$ elements. Then $\overline{I^{\ell+w}} \subseteq I^{w+1}$ for all $w \geq 0$.
\end{theorem}

If in Hochster and Huneke's Theorem \ref{HH} above, one replaces $I$ by a minimal reduction $J$, generated by $\ell$ elements (assuming that the ring $R$ is local with infinite residue field), one obtains that $\overline{I^{\ell+w}} \subseteq (J^{w+1})^*$.   The relatively simple argument that is used leads one to examine the coefficients of the elements of $J$.  For simplicity, consider the case $w=0$. Given $z \in \overline{I^\ell}= \overline{J^\ell}$, there exists an element $c \in R^0$ such that $c z^q \in (J^\ell)^q$. Since $J$ is generated by $\ell$ elements, then $cz^q \in J^{[q]}J^{(\ell-1)q}$. Further information can be obtained from taking into consideration the factor $J^{(\ell-1)q}$ and has led to results of the form $\overline{I^{\ell+w}} \subseteq J^{w+1}K$ where $I$ is an ideal of analytic spread $\ell$ in a regular local ring $R$, $J$ is a minimal reduction of $I$, and $K$ is an ideal of coefficients.

 Towards the above goal, Aberbach and Huneke introduced the following definition in \cite{AH3}:
\begin{definition} Let $R$ be a commutative Noetherian ring and let $J \subseteq I$ be two ideals of $R$. The \textit{coefficient ideal} of $I$ relative to $J$, denoted by $\mathfrak{a} (I,J)$, is the largest ideal $\mathfrak{b}$ of $R$ for which $I\mathfrak{b}=J\mathfrak{b}$.
\end{definition}

They were then able to prove Theorem \ref{coeffth1}. In the next section, we state and prove a generalization of this theorem to ideals which are not necessarily $\m$-primary.  See Theorem~\ref{mainthm} for a specific statement.\\

\section{A Brian\c{c}on-Skoda theorem with coefficients}

We are now ready to present the argument needed to generalize Theorem~\ref{coeffth1}.\\

\noindent
{\bf Notation}. If $J\subseteq I$ are two ideals of $R$, $x_1,\dots,x_n$ are elements of $R$ and $t$ is any positive integer, then $\mathfrak{a}_t$ will denote the coefficient ideal of the ideal $I+(x_1^t,\dots,x_n^t)$ relative to the ideal $J+(x_1^t,\dots,x_n^t)$. 

\begin{lemma} $(\mathfrak{a}_t)_t$ is a decreasing sequence of ideals.
\end{lemma}

\begin{proof} [{\bf Proof}] In order to prove the inclusion $\mathfrak{a}_{t+1} \subseteq \mathfrak{a}_t$, it is enough to show that the  inclusion $\mathfrak{a}_{t+1}(I+(x_1^t,\dots,x_n^t)) \subseteq \mathfrak{a}_{t+1}(J+(x_1^t,\dots,x_n^t))$ holds, since this then implies that $\mathfrak{a}_{t+1}(I+(x_1^t,\dots,x_n^t)) = \mathfrak{a}_{t+1}(J+(x_1^t,\dots,x_n^t))$.   But $\mathfrak{a}_t$ is the largest ideal for which this equality holds.

Now the inclusion $\mathfrak{a}_{t+1}(I+(x_1^t,\dots,x_n^t)) \subseteq \mathfrak{a}_{t+1}(J+(x_1^t,\dots,x_n^t))$ is easy to prove since on one hand we have $\mathfrak{a}_{t+1}I \subseteq \mathfrak{a}_{t+1}(I+(x_1^{t+1},\dots,x_n^{t+1})) = \mathfrak{a}_{t+1}(J+(x_1^{t+1},\dots,x_n^{t+1})) \subseteq \mathfrak{a}_{t+1}(J+(x_1^t,\dots,x_n^t))$, and on the other hand,  $\mathfrak{a}_{t+1}x_i^t \subseteq \mathfrak{a}_{t+1}(J+(x_i^t))\subseteq \mathfrak{a}_{t+1}(J+(x_1^t,\dots,x_n^t))$ for all $i=1,\dots,n$. 

Hence $\mathfrak{a}_{t+1} (I+(x_1^t,\dots,x_n^t)) \subseteq \mathfrak{a}_{t+1}(J+(x_1^t,\dots,x_n^t))$.
\end{proof}

\begin{lemma} \label{inc} Let $\mathfrak{a}=\mathfrak{a}(I,J)$ and $\mathfrak{b}=\displaystyle\cap_t \mathfrak{a}_t$. Then $\mathfrak{a}\subseteq \mathfrak{b}$. 
\end{lemma}

\begin{proof} [{\bf Proof}] To prove $\mathfrak{a} \subseteq \mathfrak{a}_t$ for all $t$, we show that  $\mathfrak{a}(I+(x_1^t,\dots,x_n^t)) \subseteq \mathfrak{a}(J+(x_1^t,\dots,x_n^t))$.\\
But we have that $\mathfrak{a} I = \mathfrak{a} J \subseteq \mathfrak{a} (J+(x_1^t,\dots,x_n^t))$, and also that for any $i=1,\dots,n$, $\mathfrak{a} x_i^t \subseteq \mathfrak{a} (J+(x_i^t)) \subseteq \mathfrak{a}(J+(x_1^t,\dots,x_n^t))$. 

Hence the inclusion $\mathfrak{a}(I+(x_1^t,\dots,x_n^t)) \subseteq \mathfrak{a}(J+(x_1^t,\dots,x_n^t))$ is clear.
\end{proof}

We will need Chevalley's theorem in order to prove Theorem~\ref{mainthm}.

\begin{theorem} \label{chev} \textup{(\cite{Ch}, Lemma 7)} Let $(R,\m)$ be a complete local ring and let $\{J_n\}_n$ be a decreasing sequence of ideals with $\cap_n J_n = 0$. Then, for all $n \geq 1$, there exists $t_n \geq 1$, such that $J_{t_n} \subseteq m^n.$
\end{theorem}

We now present the main theorem in this paper.

\begin{theorem}\label{mainthm} Let $(R,\m)$ be a regular local ring of dimension $d$ containing a field. Let $I$ be an ideal of $R$ of analytic spread $\ell$, and let $J$ be a reduction of $I$. Choose $x_1,\dots,x_n$ in $R$ such that the ideal $I+(x_1,\dots,x_n)$ is $m$-primary. Let $\mathfrak{b}=\displaystyle\cap_t \mathfrak{a}_t$ (with $\mathfrak{a}_t$ being as in the notation above), and assume that $R/\mathfrak{b}$ is complete (in particular $R$ itself may be complete). Then $\mathfrak{b}= \mathfrak{a}(I,J)$ and for all $w \geq -1$ we have 
$$ \overline{I^{\ell+w}} \subseteq J^{w+1} {\mathfrak{a}} (I,J).$$
\end{theorem}

\begin{proof} [{\bf Proof}] Since $J$ is a reduction of $I$, there exists $r$ such that $JI^r=I^{r+1}$ and this implies that for any ideal $L$ of $R$, $(J+L)(I+L)^r= (I+L)^{r+1}$. In fact we have: 
\begin{equation*}
\begin{split}
(I+L)^{r+1}= I^{r+1}+I^r L+\dots+L^r I + L^{r+1} & = JI^r+ L(I^r+\dots+L^{r-1}I+L^r)\\
&\subseteq (J+L)(I+L)^r \subseteq (I+L)^{r+1}.
\end{split}
\end{equation*}
In particular, for all $t$, we have $(J+(x_1^t,\dots,x_n^t))(I+(x_1^t,\dots,x_n^t))^r= (I+(x_1^t,\dots,x_n^t))^{r+1}$. Hence for all $t$, $J+(x_1^t,\dots,x_n^t)$ is a reduction of $I+(x_1^t,\dots,x_n^t)$. Now apply Theorem \ref{coeffth1} to the $m$-primary ideal $I+(x_1^t,\dots,x_n^t)$ to conclude that $$\overline{I^{\ell+w}} \subseteq \overline{(I+(x_1^t,\dots,x_n^t))^{\ell+w}} \subseteq (J+(x_1^t,\dots,x_n^t))^{w+1} \mathfrak {a}_t.$$

Next, we show that $\mathfrak{a}= \mathfrak{b}$ where $\mathfrak{a}=\mathfrak{a}(I,J)$ is the coefficient ideal of $I$ relative to $J$. We already know from Lemma \ref{inc} that $\mathfrak{a} \subseteq \mathfrak{b}$. 
If $\mathfrak{b}$ is strictly larger than  $\mathfrak{a}$, then $\mathfrak{b} J \neq \mathfrak{b} I$. Thus there are elements $y \in \mathfrak{b}$ and $c \in I$ with $y c \notin \mathfrak{b} J$.

We are going to prove that $y c\in \mathfrak{b}J$, and therefore by contradiction we conclude that $\mathfrak{b} = \mathfrak{a}$.

For any $t$, $y$ is an element of $\mathfrak{a}_t$ and this implies that $y c \in \mathfrak{a}_t I \subseteq \mathfrak{a}_t (I+(x_1^t,\dots,x_n^t)) = \mathfrak{a}_t (J+(x_1^t,\dots,x_n^t)) \subseteq \mathfrak{a}_t J + (x_1^t,\dots,x_n^t)$.
Hence, $yc \in \bigcap_{t} (\mathfrak{a}_t J + (x_1^t,\dots,x_n^t)) \subseteq \bigcap_t (\mathfrak{a}_t J+ m^t)$.

Since $R/{\mathfrak{b}}$ is complete and $(\mathfrak{a}_t)$ is a decreasing sequence with $\displaystyle\cap_t \mathfrak{a}_t= \mathfrak{b}$, Chevalley's theorem shows that for all $j \in \mathbb{N}$, there exists $t_j $ such that $\mathfrak{a}_{t_j} \subseteq \mathfrak{b}+ m^j$ and the sequence $(t_j)$ can be chosen increasing.
Consequently, we deduce that for any $t \geq t_1$, there exists $j_t \in \mathbb{N}$ with $\mathfrak{a}_t \subseteq \mathfrak{b} + m^{j_t}$, and such that the sequence $(j_t)$ is increasing to infinity. This can be done by taking $j_t = k$ for all $t_k \leq t < t_{k+1}$, $k\geq 1$. 

Hence if $t \geq t_1$, we obtain that $ \mathfrak{a}_t J + m^t \subseteq (\mathfrak{b} + m^{j_t})J + m^t  \subseteq \mathfrak{b} J + m^{\lambda}$ where $\lambda = \mbox{minimum} \{j_t, t\}$.  Note that $\lambda$ is going to infinity as t goes to infinity.
Therefore, $\bigcap_{t} (\mathfrak{a}_t J  + m^t) \subseteq \bigcap_{\lambda \to \infty} (\mathfrak{b} J + m^{\lambda}) \subseteq \mathfrak{b} J$, by the Krull intersection theorem. 

Thus we have proved that $y c \in \mathfrak{b} J$,  a contradiction. 
The desired conclusion $\mathfrak{b} = \mathfrak{a}$ now follows. 

To finish the proof of the theorem, recall that we have already proved that for all $t$, $ \overline{I^{\ell+w}} \subseteq (J+(x_1^t,\dots,x_n^t))^{w+1} \mathfrak {a}_t.$

But for $t\gg 0$, there exists $j_t$ such that $\mathfrak{a}_t \subseteq \mathfrak{b}+ m^{j_t}= \mathfrak{a}+ m^{j_t}$ and $(j_t)$ is increasing to infinity. Hence,
\begin{equation*}
\begin{split}
\overline{I^{\ell+w}} &\subseteq (J+(x_1^t,\dots,x_n^t))^{w+1} \mathfrak {a}_t\\
& \subseteq J^{w+1} \mathfrak{a}_t + (x_1^t,\dots,x_n^t)\\
& \subseteq J^{w+1} (\mathfrak{a} + m^{j_t})+m^t\\
& \subseteq J^{w+1} \mathfrak{a} + m^{\mbox{min}\{j_t,t\}}
\end{split}
\end{equation*}
where min $\{j_t, t\} \rightarrow \infty$ as $t\rightarrow \infty$.
By another application of the Krull intersection theorem we finally conclude that $\overline{I^{\ell+w}} \subseteq J^{w+1} \mathfrak{a}$, proving the theorem. 
\end{proof}

\begin{question} Can we prove Theorem \ref{mainthm} without assuming that $R/\mathfrak{b}$ is complete? We will have an affirmative answer if the coefficient ideal commutes with completion, i.e. if $\mathfrak{a}(I,J) \hat{R}= \mathfrak{a}(I\hat{R},J\hat{R})$. Because if this is true, then as $\hat{R}$ is faithfully flat, one deduces that 
\begin{equation*}
\begin{split} 
\overline{I^{\ell+w}} = \overline{I^{\ell+w}}\hat{R} \cap R & \subseteq \overline{I^{\ell+w} \hat{R}} \cap R\\
& \subseteq J^{w+1} \mathfrak{a}(I\hat{R},J\hat{R})\hat{R} \cap R\\
& = J^{w+1} \mathfrak{a}(I,J) \hat{R} \cap R \\
& = J^{w+1} \mathfrak{a}(I,J).
\end{split}
\end{equation*}
Note that we always have $\mathfrak{a}(I,J) \hat{R} \subseteq \mathfrak{a}(I\hat{R},J\hat{R})$. We would like to know whether the second inclusion holds in general. 
\end{question}

\begin{remark} We observe that the coefficient ideal does not commute with localization. Consider $J \subseteq I$ with $J_P = I_P$ for some prime $P$, but not equal up to integral closure. Replace $J$ by $m^n J$. Then for $n \gg 0$, $\overline {m^n J} \subset \overline {I}$ but are not equal. Thus $\mathfrak{a}(m^n J,I) = 0$ but $\mathfrak{a}((m^n J)_P,I_P)= \mathfrak{a}(I_P,I_P)= R_P$. \\
\end{remark}

\end{document}